\documentclass[11pt,reqno]{amsart}
\usepackage{amsfonts,amssymb,amsmath,amsopn,amsthm,graphicx}
\usepackage{amsxtra, mathrsfs}
\usepackage{colordvi}
\usepackage[usenames,dvipsnames]{color}

\usepackage[utf8]{inputenc}
\newcommand\version{\today}

\usepackage{amsmath,amsfonts,amsthm,amssymb,amsxtra}

\newtheorem{theorem}{Theorem}[section]
\newtheorem{proposition}[theorem]{Proposition}
\newtheorem{lemma}[theorem]{Lemma}
\newtheorem{corollary}[theorem]{Corollary}

\numberwithin{equation}{section}

\theoremstyle{definition}

\newtheorem{assumption}[theorem]{Assumption}

\theoremstyle{remark}

\newtheorem{remark}[theorem]{\bf Remark}

\newcommand{\A}{\mathcal{A}}
\newcommand{\B}{\mathfrak{B}}

\newcommand{\eps}{\varepsilon}
\newcommand{\e}{{\rm e}}

\newcommand{\HH}{\mathcal{H}}

\newcommand{\N}{\mathbb{N}}

\newcommand{\Q}{\mathcal{Q}}
\newcommand{\R}{\mathbb{R}}

\newcommand{\U}{\mathcal{U}}
\newcommand{\V}{\mathcal{V}}

\DeclareMathOperator{\essinf}{ess\; inf}

\DeclareMathOperator{\tr}{Tr}

\begin{document}

\title[Schr\"odinger operators with Robin boundary conditions --- \version]{Schr\"odinger operators on exterior domains with Robin boundary conditions: heat kernel estimates}

\author {Hynek Kova\v{r}\'{\i}k}

\address {Hynek Kova\v{r}\'{\i}k, DICATAM, Sezione di Matematica, Universit\`a degli studi di Brescia, Italy}

\email {hynek.kovarik@unibs.it}

\author {Delio Mugnolo}

\address {Delio Mugnolo, Lehrgebiet Analysis, FernUniversit\"at in Hagen, Germany}

\email {delio.mugnolo@fernuni-hagen.de}


\begin{abstract}
We study Schrödinger operators with Robin boundary conditions on exterior domains in $\R^d$. 
We prove sharp point-wise estimates for the associated semigroups which show, in particular, how the boundary conditions affect 
the time decay of the heat kernel in dimensions one and two. Applications to spectral estimates are discussed as well. 
\end{abstract}

\thanks{Version of \today }

\maketitle

%%%%%%%%%%%%%%%%%%%%%%%%%%%%%%%%%%%%%%%%%
\section{\bf Introduction}

In this paper we consider Laplace operators with Robin boundary conditions defined on domains of the type $M=\R^d\setminus K$, where $K\subset\R^d$ is an open bounded set. Given a bounded  function $\sigma: \partial M\to \R$ we consider the Laplace operator $-\Delta_\sigma$ in $L^2(M)$ defined by means of the sesquilinear form
\begin{equation} \label{q-form}
Q_\sigma[ u, v] = \int_{M}\,  \nabla \overline{ u} \cdot\nabla v\, dx  +  \int_{\partial M}   \sigma\, \overline{u}\, v \, dS, \qquad u,v\in H^1(M). 
\end{equation}
Note that the above form with $\sigma=0$ generates the Neumann Laplacian $-\Delta_0$ on $L^2(M)$.
The standard theory of Gaussian heat kernel estimates, see e.g.~\cite[Thms.~6.1, 6.2]{grig} or \cite[Sect.~4.2]{sc}, implies that there exist positive constants $c$ and $C >0$ such that the semigroup generated by $-\Delta_0$ satisfies
\begin{equation} \label{neumann-ub}
C^{-1}\ t^{-\frac d2}\, e^{-\frac{c |x-y|^2}{ t}} \, \leq \, \e^{ t \Delta_{0}}(x,y) \, \leq \, C\ t^{-\frac d2}\, e^{-\frac{|x-y|^2}{c t}} \qquad \forall \ x,y \in M, \quad t>0\, .
\end{equation}  
The goal of this paper is to show that if $\sigma>0$ and $d\leq 2$, then the heat kernel generated by the Robin Laplacian $-\Delta_\sigma$ decays faster than the heat kernel of the Neumann Laplacian $-\Delta_0$ and to establish sharp estimates on the decay rate. 

\smallskip

In order to quantify the effect of the boundary term in \eqref{q-form} we will work in a more general setting and consider
Schr\"odinger operators in $L^2(M)$ of the type
\[
H_\sigma(\lambda, U) = -\Delta_\sigma -\lambda\, U ,
\]
(to be interpreted in a weak sense as a form sum),
where $U:M\to \R$ is a real-valued positive function and $\lambda>0$ is a coupling constant. Under suitable conditions on $U$, see Corollary \ref{cor-1} below, the operator $-H_\sigma(\lambda,U)$ generates a semigroup on $L^2(M)$ given by an integral kernel which we denote by
\[
\e^{-t H_\sigma(\lambda, U)}(x,y), \qquad x,y \in M .
\]
We will pay particular attention to the case $d=2$ which is studied in detail in section \ref{sec-d2}. Our aim is to prove that the presence of Robin boundary conditions accelerates the decay of $\e^{-t H_\sigma(\lambda, U)}(x,y)$ in such a way that if $U>0$ belongs to a certain potential class and if $\lambda$ is small enough, then the semigroup $\e^{-t H_\sigma(\lambda, U)}$ results transient. This is in sharp contrast to the case of Neumann boundary conditions. i.e.~$\sigma=0$, where the associated semigroup $\e^{-t H_0(\lambda, U)}$ is recurrent even for $\lambda=0$ as follows from equation \eqref{neumann-ub} with $d=2$. 

The decay of the heat kernel generated by $H_\sigma(\lambda,U)$ depends, apart from the boundary conditions, also on the potential $U$. Hence in order to establish sharp heat kernel bounds we will assume that $U$ can be controlled by the reference potential 
\begin{equation} \label{ref-pot}
 U_{\sigma}(x) := \frac{1}{4 |x|^2} \left( \log \frac{|x|}{\rho} +\frac{1}{\rho\sigma_0}\right)^{-2}\, ,
\end{equation} 
where $\rho$ is the in-radius of $K$ and  $\sigma_0$ is the essential infimum of $\sigma$. More precisely, we will show that if $U\leq U_\sigma$  and if $\lambda\leq 1$, then the heat kernel satisfies 
\begin{equation} \label{t-infty}
\e^{-t H_\sigma(\lambda, U)}(x,y) = \mathcal{O}\left(t^{-1}\, ( \log t)^{-1-\sqrt{1-\lambda}}\right)  \qquad t\to \infty, 
\end{equation}
point-wise for all $x,y\in M$, see Theorem \ref{thm-upperb} for details. The logarithmic factor, which makes the heat kernel decay faster with respect to \eqref{neumann-ub}, reflects the effect of the boundary conditions. On the other hand, the presence of the negative potential $-\lambda U$ is reflected by the term $\sqrt{1-\lambda}$ in the power of the logarithm. 

Similarly, if $U\geq U_\sigma$ then the heat kernel is bounded below by a function which has the same decay in $t$ as the right hand side of equation \eqref{t-infty}, see Proposition \ref{prop-lowerb}. In other words, the decay rate in $t$ in estimate \eqref{t-infty} is sharp. A two-sided estimate on the heat kernel in the case $U=U_\sigma$ is established in Theorem \ref{cor-2-sided}. The latter implies, in particular, that the semigroup $\e^{-t H_\sigma(\lambda, U_\sigma)}$ is transient for $\lambda < 1$ and recurrent for $\lambda=1$. 

\smallskip

Operators with Dirichlet boundary conditions at $\partial M$ are discussed in section \ref{sec-dirichlet}, see Theorem \ref{cor-dirichlet}. We use the fact that Dirichlet boundary conditions can be achieved as a limiting case of the Robin ones by changing the form domain in \eqref{q-form} to $H^1_0(M)$ and subsequently letting $\sigma \to+\infty$. The reference potential \eqref{ref-pot} then takes the form
\begin{equation} \label{u0}
U_\infty(x): =  \frac{1}{4|x|^2}\left( \log \frac{|x|}{\rho} \right)^{-2} .
\end{equation}
For heat kernel estimates of Dirichlet Laplacians, without an additional negative potential, in unbounded domains, and in particular in exterior domains, we refer to \cite{gs2, zh1, zh2}. 

\smallskip

The proof of our main results relies upon transforming the problem to an analysis of a Neumann Laplacian in suitable weighted $L^2-$spaces with a $\lambda$-dependent weight. We then employ the technique of the Li-Yau type heat kernel estimates on weighted manifolds invented by Grigor'yan and Saloff-Coste, see  \cite{grig, gs2} or \cite[Chap.~4]{sc} and references therein. In section \ref{sec-appl} we discuss some applications of the obtained heat kernel bounds to Hardy and Hardy-Lieb-Thirring inequalities for Schr\"odinger operators on exterior two-dimensional domains.  

\smallskip

Although we are primarily interested in semigroups generated by Robin Laplacians in dimension two, we discuss the analogous problem in other dimensions as well. It turns out that while the effect of the boundary on the decay rate of the associated heat kernel is even stronger in dimension one, see Theorem \ref{thm-1d}, in dimensions larger than two it is absent. Although the latter assertion is well-known for Dirichlet heat kernels, \cite{gs2}, for the sake of self-containdness we state an analogous result for Robin Laplacians in Proposition \ref{prop-hd}. 

%%%%%%%%%%%%%%%%%%%%%%%%%%%%%%%%%%%%%%%%%%%%%
\section{\bf The case $d=2$}
\label{sec-d2}

\noindent Throughout Sections~\ref{sec-d2} and~\ref{sec-appl} we will work under the following conditions on the potential $U$, the exterior domain $M$ and the coefficient $\sigma$ in the Robin boundary conditions.

\begin{assumption} \label{ass-U}
There exists $p\in [2,\infty)$ such that $U\in L^p(\R^2)+L^\infty(\R^2)$. In other words $U=U_1+U_2$ with $U_1\in L^p(\R^2),\, p\geq 2,$ and $U_2\in L^\infty(\R^2)$. 
\end{assumption}

\begin{assumption}\label{ass-M}
The set $K\subset\R^2$ is open, bounded and simply connected with Lipschitz boundary; we let $M:=\R^2\setminus K$.
\end{assumption}

\begin{assumption}\label{ass-sigma}
The coefficient $\sigma$ lies in $L^\infty(\partial M)$ and we denote by
$$
\sigma_0 : = \essinf\limits_{\partial M}  \ \sigma
$$
its essential infimum on $\partial M$. 
\end{assumption}

\subsection{Preliminaries}
\label{sec-prelim}

\begin{lemma} \label{lem-domain}
Assume $ \sigma$ to be non-negative.
Then  the sesquilinear form 
\begin{equation} \label{q-form-lambda}
\widetilde{Q}_{\sigma,\lambda}[ u, v] = \int_{M}\, \nabla \overline{ u} \cdot\nabla v\, dx  - \lambda \int_M \, U \overline{u} v \, dx +  \int_{\partial M}   \sigma\, \overline{u}\, v \, dS 
\end{equation}
defined on $H^1(M)\times H^1(M)$ is closed in $L^2(M)$ for all $\lambda\in\R$.
\end{lemma}

\begin{proof}
Let $u,v \in H^1(M)$ and moreover let $p\geq 2$ be as in assumption \ref{ass-U}.
Under the regularity assumptions on $K$ it follows from standard Sobolev imbedding theorems and trace inequalities, see e.g.~\cite[Thm.~4.12 and Thm.~5.36]{ad} that there exists a constant $C_q>0$ such that 
\begin{equation} \label{imbed}
\| f\|_{L^q(\partial M)} \, \leq\, C_q \, \| f\|_{H^1(M)} \, , \qquad \| f\|_{L^q(M)}\, \leq \, C_q\, \| f\|_{H^1(M)} 
\end{equation}
hold true for all $f\in H^1(M)$, all $q\in [2,\infty)$. Hence by the assumption on $U$ and H\"older inequality we have
\begin{align*} 
\Big | \int_{\partial M}   \sigma\, u\, v \, dS  \Big | &  \, \leq\, \|\sigma\|_\infty \, C_M^2 \, \| u\|_{H^1(M)} \,  \| v\|_{H^1(M)}  \\
\Big | \int_M \, U u v \, dx \Big | & \, \leq \, C_q \|U_1\|_p\, \| u\|_{H^1(M)}\, \| v\|_{H^1(M)}  + C_2 \|U_2\|_\infty\, \| u\|_{H^1(M)}\, \| v\|_{H^1(M)}\, . 
\end{align*}
This shows that 
\begin{equation} \label{form-upperb}
\big | \widetilde{Q}_{\sigma,\lambda}[ u, v] \big | \, \leq \, C\, \| u\|_{H^1(M)}\, \| v\|_{H^1(M)} \, .
\end{equation}
In order to prove a suitable lower bound on $\widetilde{Q}_{\sigma,\lambda}[v,v]$ we recall the Gagliardo-Nirenberg inequality
\begin{equation} \label{mazya}
\|f\|_q \ \leq \ K_q \, \|f\|_2^{\frac 2q} \left( \|\nabla f\|_2 +|f\|_2\right) ^{1-\frac 2q}\, \qquad \forall\ f\in H^1(M)
\end{equation} 
which holds for all $q\in [2,\infty)$, see e.g.~\cite[Thm.~5.8]{ad}. This and the Young inequality:
\begin{equation} \label{young}
A B \, \leq\,  \frac{\delta^r}{r}\, A^r + \frac{\delta^{-r'}}{r'}\ B^{r'}, \qquad A,B>0,  \quad \frac 1r +\frac{1}{r'} =1, \quad \delta>0, 
\end{equation}
implies that for every $\eps>0$ there exists a constant $c_1(\eps)$ such that 
\begin{equation} \label{eps-1}
\|f\|^2_q \, \leq \, \eps \|\nabla f\|^2_2+ c_1(\eps)\, \|f\|^2_2\,  \qquad \forall\ f\in H^1(M). 
\end{equation} 
Thus, similarly as above, we can use the H\"older inequality to get 
\begin{align*}
\Big | \int_M \, U v^2 \, dx \Big |  & \, \leq \,  \|U_1\|_p\, \|v\|^2_q + \|U_2\|_\infty\, \|v\|_2^2 \\
& \leq  \,   \eps\, \|U_1\|_p\,  \|\nabla v\|^2_2+ \left( \|U_1\|_p\, c_1(\eps)+\|U_2\|_\infty \right)\, \|v\|^2_2
\end{align*}
Now if we choose $\eps$ small enough, then we conclude that the lower bound 
$$
\widetilde{Q}_{\sigma,\lambda}[v,v] + \|v\|^2_2\, \geq\, c\, \|v\|_{H^1(M)}\, ,
$$
holds for some $c>0$ and all $v\in H^1(M)$. In view of  \eqref{form-upperb} this completes the proof.
\end{proof}

\smallskip

\noindent In the sequel we denote by $H_\sigma(\lambda,U)$ the unique self-adjoint and positive operator on $L^2(M)$ associated with the sesquilinear form $\widetilde{Q}_{\sigma,\lambda}[\, \cdot, \cdot]$. As a consequence of the above lemma we obtain

\begin{corollary} \label{cor-1}
If $ \sigma$ is non-negative, then the operator $H_\sigma(\lambda,U)$ generates on $L^2(M)$ a sub-Markovian semigroup given by an integral kernel. 
\end{corollary}

\begin{proof}
The form $\widetilde{Q}_{\sigma,\lambda}$ is symmetric and, by Lemma \ref{lem-domain}, also closed. Moreover, a direct computation shows that the Beurling-Deny conditions are satisfied, hence the operator $-H_\sigma(\lambda,U)$ generates on $L^2(M)$ an analytic sub-Markovian semigroup. By the Sobolev embedding theorem, cf.~\cite[Cor.~9.14]{Bre10}, $H^1(M)\hookrightarrow L^q(M)$ for all $q\in [2,\infty)$. Thus, the semigroup is ultracontractive and is hence given by an integral kernel of class $L^\infty(M\times M)$, see e.g.~\cite[\S~7.3.2--7.3.3]{Are04}.
\end{proof}

\subsection{Notation}  

In the sequel we denote by $B(x,r)\subset \R^d$ a ball of radius $r$ centered in $x$. 
Let $\rho>0$ be the in-radius of $K$: 
\begin{equation} \label{rho} 
\rho = R_{\rm in}(K):= \sup_{y\in K} {\rm dist}(y, \partial K).
\end{equation} 
Without loss of generality we may choose the coordinate system in such a way that 
\begin{equation} \label{inner-ball}
B(0,\rho) \subseteq K. 
\end{equation}
We denote by $-\Delta_D$ the Dirichlet Laplacian in $L^2(M)$.

\smallskip

%%%%%%%%%%%%%%%%%%%%%%%%%%%%%%%%%%%%%%%%%%%
\subsection{Heat kernel upper bounds} 
\label{sec-upperb}

Throughout this section, our techniques rely upon the assumption that the parameter $\sigma_0$ introduced in Assumption~\ref{ass-sigma} satisfies
\[
\sigma_0>0\ .
\]
We have

\begin{theorem} \label{thm-upperb}
In addition to the Assumption~\ref{ass-M}, let $K\subset\R^2$ have $C^{2}$-regular boundary. 
Let $\rho>0$ be given by \eqref{rho} and let $\sigma_0>0$. 
Suppose moreover that
\begin{equation} \label{cond-v}
U(x) \, \leq \,  \frac{1}{4 \, |x|^2} \left( \log \frac{|x|}{\rho} +\frac{1}{ \rho\,  \sigma_0 }\right)^{-2} \qquad \forall\ x\in M ,
\end{equation}
and that $\lambda \in [0,1]$. Then
there exist positive constants $c, C$ such that for all $x,y\in M$ and all $t>0$ 
\begin{equation} \label{hk-upperb-2d}
 \e^{-t H_\sigma(\lambda, \, U)}(x,y) \ \leq \ \frac{C \left(\log\frac{|x|}{\rho} + \frac{1}{\rho\, \sigma_0 }\right)^{\frac{1+\sqrt{1-\lambda}}{2}}\, \left(\log\frac{|y|}{\rho} + \frac{1}{\rho\, \sigma_0 }\right)^{\frac{1+\sqrt{1-\lambda}}{2}}
 \ \e^{-\frac{|x-y|^2}{c\, t}}}{t \left(\log\left(\frac{|x| +\sqrt{t}}{\rho} \right)+ \frac{1}{\rho\, \sigma_0 }\right)^{1+\sqrt{1-\lambda}}}\ .
\end{equation}
\end{theorem}

\smallskip

\begin{remark}
The condition $\lambda\leq 1$ is necessary. Indeed, if $K$ is a ball, then the operator $H_\sigma(\lambda,  U_\sigma)$ is not positive for $\lambda>1$, see Proposition \ref{prop-hardy}.
\end{remark}

\begin{proof}[Proof of Theorem \ref{thm-upperb}]
Our assumptions on $K$ imply that there exists a mapping $N:[0,2\pi] \to \N$ and $C^{2}$-regular functions $R_j, S_j : [0,2\pi] \to [0,\infty]$  such that
\begin{equation} 
R_j(\theta) \leq S_{j+1}(\theta) \leq R_{j+1}(\theta) \qquad \forall\ j= 0,1, \dots,N(\theta)\, , \quad \forall\ \theta\in [0,2\pi]\, ,
\end{equation}
and 
\begin{equation}  \label{K-repr}
K= \Bigg\{ (r \cos\theta, r\sin\theta)\, : \, \theta\in[0,2\pi], \  r\in \big [0, R_0(\theta)\big)\, \cup  \bigcup_{j=1}^{N(\theta)}\big(S_j(\theta),\, R_j (\theta) \big) \Bigg\}.
\end{equation}
In particular, $K$ is star-shaped if and only if $N(\theta) =0$ for all $\theta\in [0,2\pi]$.
Moreover, by assumption the curvature of $\partial K$ is bounded. Hence 
\begin{equation} \label{N-finite}
\sup_{\theta\in[0,2\pi]} \, N(\theta) < \infty.
\end{equation}
%\footnote{\textcolor{red}{Delio 23.12.15: L'intuizione geometrica \`e abbastanza chiara, ma non so se anche dimostrarlo formalmente \`e cos\'i %facile. Hai per caso una citazione precisa a portata di mano? 
%Delio 9.1.16: Forse ci sono, ho trovato la soluzione negli appunti del corso di geometria differenziale della mia università :)
%Se una curva chiusa $c:[a,b]\to \mathbb R^2$ è $C^2$, come il bordo del nostro $M$, allora esiste una $\theta\in C^1([a,b];\mathbb R)$ tale %che la funzione angolare $c'$ della curva soddisfi
%$c'(t)=|c'(t)|(cos(\theta(t))e_1 + \sin(\theta(t)) e_2)$, $t\in [a,b]$; la curvatura della curva è allora $\theta'$, che ovviamente è limitata essendo %continua su $[a,b]$. Credo che questo sia esattamente quello che cercavamo, no?}}
Next we define the weight function
\begin{equation} \label{eq-w}
w(x) := \left(\log\frac{|x|}{\rho} + \beta\right)^\alpha, \qquad x\in M,
\end{equation}
where $\alpha\in\R$ and $\beta\in\R$ are two positive parameters whose values will be specified later. 
Since $w$ is positive on $M$, we can write any test function $u\in H^1(M)$ as a product  
\begin{equation} \label{u-subst} 
u(x) = w(x)\, f(x), 
\end{equation}  
for some
\begin{equation} 
f\in H^1(M, w^2 dx) := \left\{ f\in H^1(M)\, :\,   \int_M (|\nabla f|^2+|f|^2)\, w^2\, dx < \infty  \right\}. 
\end{equation}
Let $A_\sigma(\lambda,U)$ be the self-adjoint operator in $L^2(M, w^2 dx)$ associated with the closed quadratic form
\begin{equation} \label{q-from-w}
Q_\sigma[w f, w f] - \lambda \int_M U |f|^2\, w^2 dx, \qquad f\in  H^1(M, w^2 dx),
\end{equation}
which by the Beurling-Deny criteria generates a sub-Markovian semigroup on $L^2(M,w^2 dx)$. As already mentioned, in order to show that this semigroup has a kernel it suffices to prove its ultracontractivity; by~\cite[Thm.\ in \S~7.3.2]{Are04} this is in turn equivalent to showing that the  imbedding $H^1(M,w^2 dx)\hookrightarrow L^\frac{2m}{m-2}(M,w^2 dx)$ is continuous for some $m>2$ -- this is done in Lemma~\ref{lem-sobolev}.
Let now $\e^{-t A_\sigma(\lambda,U)}(x,y)$ be the integral kernel of this semigroup.
Note that the mapping $f\mapsto w\, f$ is an isometry from $L^2(M, w^2dx)$ onto $L^2(M,dx)$. Hence in view of \eqref{u-subst} it follows that 
\begin{equation} \label{eq-kernels}
 \e^{-t H_\sigma(\lambda, U)}(x,y) = w(x)\, w(y)\, \e^{-t A_\sigma(\lambda,U)}(x,y),\qquad \forall\ x,y\in M, \quad t>0.
\end{equation}
We can write the sesquilinear form $Q_\sigma[u,v]$ in polar coordinates as
\begin{align} \label{Q-polar}
Q_\sigma[u, v] =& \int_0^{2\pi}  \, \sum_{j=1}^{N(\theta)}\, \int_{S_j(\theta)}^{R_{j}(\theta)} \left(\, \partial_r\overline{ u} \ \partial_r v + r^{-2}\ \partial_\theta\overline{ u}\, \partial_\theta v \right)\, r dr  d\theta  \nonumber \\
& +  \int_0^{2\pi}\, \sum_{j=1}^{N(\theta)} \, \sigma(R_j(\theta), \theta) \ (\overline{u}\, v)(R_j(\theta), \theta) \, \sqrt{R^2_j(\theta)+(R'_j(\theta))^2}\  d\theta \nonumber \\
& +  \int_0^{2\pi}\, \sum_{j=1}^{N(\theta)} \, \sigma(S_j(\theta), \theta) \ (\overline{u}\, v)(S_j(\theta), \theta) \, \sqrt{S^2_j(\theta)+(S'_j(\theta))^2}\  d\theta\ .
\end{align}
Let us factorize $u,v$ as $u=w f$ and $v=w\, g$, with $f,g\in H^1(M, w^2 dx)$. 
Now assume that $f,g\in H^1(M, w^2 dx)$ are real and positive; we are going to show that
\begin{equation} \label{form-dominance}
Q_\sigma[w f, w g] -  \lambda \int_M U f g\, w^2 dx\, \geq\, \int_M  \nabla f \cdot \nabla g\ w^2 dx = : \widehat Q_\sigma[f,g] \ .
\end{equation}
Since $w$ is radial, cf.~\eqref{eq-w}, we have $ \partial_\theta (w f) = w\, \partial_\theta f$ and $\partial_\theta (w g)=w\, \partial_\theta g$. On the other hand, for the radial derivatives we obtain
\begin{align*}
\partial_r u\, \partial_r v =  & \left(\log\frac{r}{\rho} + \beta\right)^{2\alpha}\, \partial_r f\, \partial_r g + \frac{\alpha^2}{r^2} \, \left(\log\frac{r}{\rho} + \beta\right)^{2\alpha-2}\, f g \\
&+ \frac \alpha r\, (f\, \partial_r g+g\, \partial_r f) \left(\log\frac{r}{\rho} + \beta\right)^{2\alpha-1}.
\end{align*}
Next we use the shorthands $R_j(\theta)=R_j,\, S_j(\theta)=S_j$ and integrate the last term by parts with respect to $r$. This gives
\begin{align} \label{radial-der}
\int_{S_j}^{R_j}  \partial_r u\, \partial_r v\ r dr =&  \int_{S_j}^{R_j}\! \left(\log\frac{r}{\rho} + \beta\right)^{2\alpha} \partial_r f\, \partial_r g  \, r dr   + \alpha\! \left(\log\frac{R_j}{\rho} + \beta\right)^{2\alpha-1}\! (fg)(R_j, \theta) \nonumber \\
&- \alpha \left(\log\frac{S_{j}}{\rho} +\beta\right)^{2\alpha-1}\! (fg)(S_j, \theta)\\
 &+ (\alpha-\alpha^2)\! \int_{S_j}^{R_j}\! r^{-1}\left(\log\frac{r}{\rho} + \beta\right)^{2\alpha-2} f g\, dr .\nonumber
\end{align}
We emphasize that this formula is valid for all $\alpha,\beta\in (0,\infty)$.  
Let us now fix the parameters $\alpha,\beta$: we take
\begin{equation} \label{a-b}
\alpha= \frac{1+\sqrt{1-\lambda}}{2}\ , \qquad \beta =  \frac{1+\sqrt{1-\lambda}}{2 \rho \sigma_0}\ ,
\end{equation}
and plug \eqref{radial-der} into \eqref{Q-polar}. Keeping in mind the upper bound \eqref{cond-v} and the fact that by~\eqref{inner-ball}
$$
R_j(\theta) \geq \rho \qquad  \forall\, j=0,\dots N(\theta), \quad \forall\ \theta\in[0, 2\pi] , 
$$
we conclude that the inequality~\eqref{form-dominance} 
holds for all positive functions $f,g\in H^1(M, w^2 dx)$. 

Denote by $\widehat A_\sigma$ the self-adjoint operator in $L^2(M, w^2 dx)$ associated with the sesquilinear form $\widehat Q_\sigma$ with form domain $H^1(M, w^2 dx)$. 
The operator $\widehat A_\sigma$ acts on its domain, in the sense of distributions, as 
\begin{equation} \label{A-hat-sigma}
\widehat A_\sigma\, u\, = w^{-2}\  \nabla \cdot (w^2\, \nabla u).
\end{equation}
It is easy to see that if a real-valued function $f$ lies in $H^1(M,w^2 dx)$, then so do its positive part $f^+$ and the function $f\wedge 1$ and in particular both Beurling-Deny conditions are satisfied and accordingly the generated semigroups $\e^{-t A_\sigma(\lambda,U)}$ and $\e^{-t \widehat A_\sigma}$ are sub-Markovian, \cite[Cor.~2.18]{ouh}. Moreover the form domains of  $A_\sigma(\lambda,U)$ and $\widehat A_\sigma$ coincide. Hence in view of \eqref{form-dominance} we can apply \cite[Thm.~2.24]{ouh} which implies that the semigroup generated by $-A_\sigma(\lambda,U)$ is dominated by the semigroup generated by $- \widehat A_\sigma$ and hence
\begin{equation} \label{group-dominance}
\e^{-t A_\sigma(\lambda,U)}(x,y) \, \leq \, \e^{-t \widehat A_\sigma}(x,y) \qquad \forall\ x, y\in M, \quad \forall\ t>0.
\end{equation} 
Now consider the weighted manifold $(M, w^2 dx)$ endowed with the Euclidean metric. For any $x\in M$ we denote by $B(x, \sqrt{t})$ the ball of radius $\sqrt{t}$ centered in $x$. Let
\begin{equation} \label{volume}
\V_2(x,\sqrt{t} ) : = \int_{B(x, \sqrt{t})\cap M} \, w^2(y)\, dy
\end{equation}
be the volume, in $(M, w^2 dx)$, of the intersection of $B(x, \sqrt{t})$ with $M$. 
Given any  $x_0\in M$ it is easily verified that the pointed manifold $(M, x_0)$ satisfies the condition of relatively connected annuli, see e.g.~\cite[Def.~2.10]{gs2}. Moreover, in view of \eqref{eq-w} there exists a constant $C_h$, independent of $\lambda$ and $\sigma$,  such that 
\begin{equation}
\sup_{x\in B(x_0, 2r)}\, w(x) \ \leq \ C_h \inf_{x \in M\setminus B(x_0, r)}\, w(x) 
\end{equation} 
holds for all $r$ large enough, see Lemma \ref{lem-harnack}. We may thus apply \cite[Thm.~2.11]{gs2} with $d\mu = dx$ and $d\nu = w^2 dx$, which implies that $(M, w^2 dx)$ satisfies the parabolic Harnack inequality. In view of \cite[Thm.~2.8]{gs2} this further yields the following two-sided estimate on the heat kernel of $\widehat A_\sigma$;
\begin{equation} \label{ahat-two-sided}
\frac{C_1\, \e^{-\frac{c\, |x-y|^2}{ t}}}{\V_2(x,\sqrt{t} )} \, \leq \, \e^{-t \widehat A_\sigma}(x,y) \, \leq \, \frac{C_2\, \e^{-\frac{|x-y|^2}{c\, t}}}{\V_2(x,\sqrt{t} )}\, ,\qquad x, y\in M, \quad  C_1,\,  C_2, \,  c >0.
\end{equation}
Since 
\begin{equation}  \label{ball-volume-est}
\V_2(x,\sqrt{t} ) \ \geq \, c_0\ t\, \left(\log\left(\frac{|x|+\sqrt{t}}{\rho}  \right)+\beta\right)^{2\alpha}
\end{equation}
by Lemma \ref{lem-aux}, see Appendix A,  equations  \eqref{a-b} and \eqref{ahat-two-sided} in combination with \eqref{eq-kernels} and \eqref{group-dominance} imply the claim. 
\end{proof}

\begin{remark} 
Since the weighted manifold $(M, w^2 dx)$ satisfies the volume doubling property, cf.~Lemma \ref{lem-vd},  the denominator on the right hand side of \eqref{hk-upperb-2d} may be replaced by either
$$
t \left(\log\left(\frac{|x| +\sqrt{t}}{\rho} \right)+ \frac{1}{\rho \sigma_0 }\right)^{\frac{1+\sqrt{1-\lambda}}{2}}\, \left(\log\left(\frac{|y| +\sqrt{t}}{\rho} \right)+ \frac{1}{\rho \sigma_0 }\right)^{\frac{1+\sqrt{1-\lambda}}{2}}\, ,
$$
or
$$
t \left(\log\left(\frac{|y| +\sqrt{t}}{\rho} \right)+ \frac{1}{\rho \sigma_0 }\right)^{1+\sqrt{1-\lambda}}
$$
This follows from \cite[Lem.~2.4, Rem.~2.7]{gs2} and Lemma \ref{lem-aux}.
\end{remark}

\begin{remark}

Semigroups generated by Schr\"odinger operators  
$$
-\Delta + Q \qquad \text{in} \quad  L^2(\R^d)
$$
were studied by several authors. Potentials which satisfy $Q(x) = -c |x|^{-2}$ outside a compact set were considered by Grigor'yan in 
\cite[Sec.~10.4]{grig} for $d\geq 2$. The case $Q= -c|x|^{-2}$ with $c\leq \frac{(d-2)^2}{4}$ and $d\geq 3$ was treated later in \cite{ms1,ms2}. 
In both cases it was proved that the decay rate of the heat kernel depends on $c$.

On the other hand, compactly supported positive potentials $Q$ were considered by Murata four $d=2$, see \cite{m84}. He showed in particular that if $Q$ is H\"older continuous then
\begin{equation} \label{eq-murata}
\e^{-t (-\Delta +Q)}(x,y)  \ \asymp \ \frac{\varphi(x)\, \varphi(y)}{t\, (\log t)^2}\, , \qquad t\to\infty,
\end{equation}
where the function $\varphi$ satisfies $\varphi(x) = \log |x|(1+o(1))$ as $x\to\infty$. This is compatible with \eqref{hk-upperb-2d} for $\lambda=0$. It is also interesting to notice that the same point-wise decay as in \eqref{eq-murata} was observed for magnetic Laplace operators in $\R^2$ associated with radial magnetic fields of zero integral mean, see \cite{ko}. 

Finally we point out that heat kernel upper bounds for elliptic operators with (nonlocal) Robin-type boundary conditions on bounded domains were recently obtained in \cite{gmn,gmno}. 
\end{remark}

%%%%%%%%%%%%%%%%%%%%%%%%%%%%%%%%%%%%%%
\subsection{Dirichlet boundary conditions} \label{sec-dirichlet}

A straightforward  modification of Lemma \ref{lem-domain} shows that the sesquilinear form 
\begin{equation} \label{form-dir}
Q_\infty[u, v] -\lambda \int_\Omega U\, \overline{u}\, v\, dx  = \int_{M}\, \nabla \overline{ u} \cdot\nabla v\, dx- \lambda \int_\Omega U\,  \overline{u}\, v\,  dx\ ,   \qquad u,v\in H^1_0(M)
\end{equation}
is closed in $L^2(M)$ whenever $U$ satisfies assumption \ref{ass-U}. Let $H_D(\lambda, U)$ be the self-adjoint operator in $L^2(M)$ associated with the form \eqref{form-dir}.
Hence $H_D(\lambda, U)$ is subject to Dirichlet boundary conditions at $\partial M$. Our next result provides an upper bound on the semigroup generated by $H_D(\lambda, U)$.

\begin{theorem} \label{cor-dirichlet}
Let $\lambda\in [0, 1]$. 
Then there exist positive constants $C, c$ such that for all $t>0$ and all $x,y\in M$ we have
\begin{equation} \label{dirichlet-ub}
\e^{-t H_D\left(\lambda, U_\infty\right)}(x,y) \ \leq \  \frac{C\,\left(\log\frac{|x|}{\rho} \right)^{\frac{1+\sqrt{1-\lambda}}{2}}\, \left(\log\frac{|y|}{\rho} \right)^{\frac{1+\sqrt{1-\lambda}}{2}} \ \e^{-\frac{|x-y|^2}{c\, t}}}{ t \left(\log \frac{|x| +\sqrt{t}}{\rho} \right)^{1+\sqrt{1-\lambda}}}
\end{equation}
\end{theorem}

\begin{remark} Note that the potential $U_\infty$ defined in \eqref{u0} 
belongs to $L^p(M)$ for any $p\in [1,\infty) $ and therefore satisfies assumption \ref{ass-U}. 
\end{remark}

The heat kernel estimate in~\eqref{dirichlet-ub} is the precise counterpart of the estimate \eqref{hk-upperb-2d} \emph{as the Robin boundary conditions of $-\Delta-\lambda U$ tend to the Dirichlet ones}.i.e.~as $\sigma_0\to +\infty$. This is not yet a precise argument, but our proof will actually refine this observation.

\begin{proof}[Proof of Theorem \ref{cor-dirichlet}]
We consider the sequence of quadratic forms 
\begin{equation}
q_n[u,u] = \int_{M}\, |\nabla u|^2\, dx -\lambda \int_M \, U_{\frac 1n} \, |u|^2 \, dx , \qquad u\in H^1_0(M)\, ,
\end{equation}
and the corresponding self-adjoint operators $h_n(\lambda)$ in $L^2(M)$ associated with $q_n$. Then 
$$
h_n(\lambda) \geq h_{n+1}(\lambda) \geq H_D(\lambda, U_\infty)\, \geq 0, \qquad \forall\ n\in\N,
$$
where the last inequality follows from Corollary \ref{cor-hardy} below. Since 
$$
\lim_{n\to\infty} q_n[u,u] = Q_\infty[u, u] -\lambda \int_\Omega U_\infty \, |u|^2\, dx \qquad \forall\ u\in H^1_0(M),
$$
by the monotone convergence theorem, we conclude that $h_n(\lambda)$ converges to $H_D(\lambda, U_\infty)$ in the strong resolvent sense, see e.g.~\cite[Thm.~1.2.3]{da2}. Hence for each $t>0$ the semigroup $\e^{-t h_n(\lambda)}$ converges strongly to $\e^{-tH_D(\lambda,U_\infty)}$ as $n\to\infty$. On the other hand, the domination of semigroups and Theorem \ref{thm-upperb} imply that
\begin{align*}
\e^{-t h_n(\lambda)}(x,y) & \leq \e^{-t H_{\frac{n}{\rho}}\left(\lambda, \, U_{1/n}\right)}(x,y) \leq  \frac{C\, \left(\log\frac{|x|}{\rho} + \frac{1}{n }\right)^{\frac{1+\sqrt{1-\lambda}}{2}}\, \left(\log\frac{|y|}{\rho} + \frac{1}{n}\right)^{\frac{1+\sqrt{1-\lambda}}{2}}
 \ \e^{-\frac{|x-y|^2}{c\, t}}}{ t \left(\log\left(\frac{|x| +\sqrt{t}}{\rho} \right)+ \frac{1}{n }\right)^{1+\sqrt{1-\lambda}}}
\end{align*}
holds for all all $t>0$ and all $x,y\in M$. Hence  follows by passing to the limit $n\to\infty$ we conclude that \eqref{dirichlet-ub} holds almost everywhere in $M$. The continuity of $\e^{-t H_D\left(\lambda, U_\infty\right)}(x,y)$ with respect to $x,y$ then implies \eqref{dirichlet-ub} or all $x,y\in M$. 
\end{proof}

\begin{remark} Consider the case  of the two-dimensional unit ball $K=B(0,1)$. If we put $\lambda=0$, then $H_D\left(0, U_\infty\right)$ coincides with the pure Dirichlet Laplacian $-\Delta_D$ and our upper bound \eqref{dirichlet-ub} gives 
$$
\e^{t \Delta_D}(x,x) \ \leq \  \frac{c\,  \log^2 |x|}{t\, \left(\log (|x| +\sqrt{t}\, ) \right)^2}\, , \quad c>0, 
$$
which agrees with the two-sided estimate 
\begin{equation}  \label{gsc-eq} 
\frac{ \log^2 |x|}{C\, t\, \left(\log (1 +\sqrt{t}\, ) + \log |x|) \right)^2}\ \leq \ \e^{t \Delta_D}(x,x) \ \leq \ \frac{C\, \log^2 |x|}{ t\, \left(\log (1 +\sqrt{t}\, ) + \log |x|) \right)^2}\,
\end{equation}
obtained in \cite[Eq.~(1.8)]{gs2} for $|x|$ large enough. To see this we note that
$$
\frac 12 \left( \log (1 +\sqrt{t}\, ) +\log |x|\right) \, \leq\, \log (|x| +\sqrt{t}\, ) \, \leq\, \log (1 +\sqrt{t}\, ) +\log |x|
$$
holds for all $x\in M$ and $t>0$. Indeed, since $|x|>1$, we have  
\begin{align*}
2 \log (|x| +\sqrt{t}\, )   & =   \log ( |x|^2 + 2 |x|\, \sqrt{t} + t) \geq \log(|x| +|x|\, \sqrt{t}\, ) \, ,\\
&= \log (1 +\sqrt{t}\, ) +\log |x|, 
\end{align*}
and on the other hand
$$
\log (1 +\sqrt{t}\, ) +\log |x|  = \log(|x| +|x|\, \sqrt{t}\, ) \geq \log (|x| +\sqrt{t}\, ) \, .
$$
Hence the factor $\log (1 +\sqrt{t}\, ) + \log |x|$ in \eqref{gsc-eq} can be replaced by $\log (|x| +\sqrt{t}\, )$. 
\end{remark}

\smallskip

%%%%%%%%%%%%%%%%%%%%%%%%%%%%%%%%%%%%%%%%%%%
\subsection{Heat kernel lower bounds} 

In order to establish a lower bound on the heat kernel of $H_\sigma(\lambda,U)$ we obviously need a lower bound on the potential $U$. We will thus assume that 
\begin{equation} \label{U-lowerb}
U(x) \geq  \frac{1}{4 |x|^2} \left( \log \frac{|x|}{\rho} +\beta \right)^{-2} ,
\end{equation}
holds for some $\beta>0$ and all $x\in M$. Moreover, for a given $\delta>0$ we introduce the external $\delta$-neighborhood \begin{equation} \label{K-eps}
K_\delta: = \left\{ \, x\in M\ :\ {\rm dist}(x,K) <\delta \right\}
\end{equation}
of $K$.

\begin{proposition} \label{prop-lowerb}
Let $K$ have a $C^2$-boundary and assume that $U$ satisfies \eqref{U-lowerb} for some $\beta>0$. Let $0\leq \lambda < 1$. Then there exist $\eps>0$ and $c,C>0$ such that
\begin{align} \label{hk-lowerb}
  \e^{-t H_\sigma(\lambda, U)}(x,y) &\ \geq \    \frac{C\, \left(\log\frac{|x|}{\rho} + \beta\right)^{\frac{1+\sqrt{1-\lambda}}{2}}\, \left(\log\frac{|y|}{\rho} + \beta\right)^{\frac{1+\sqrt{1-\lambda}}{2}} \  \e^{-\frac{c\, |x-y|^2}{t}}}{t \left(\log\left(\frac{|x|+ \sqrt{ t} }{\rho} \right)+\beta \right)^{1+\sqrt{1-\lambda}}} 
\end{align}
holds for all $x,y\in M\setminus K_\eps$ and all $t>0$.
\end{proposition}

\begin{proof}
Let $\Omega = M \setminus \partial K$.
By domination of semigroups
\begin{equation} \label{lowerb-1}
 \e^{-t H_\sigma(\lambda, U)}(x,y)  \, \geq \,  \e^{-t H_D(\lambda, U)}(x,y) \qquad \forall\ x, y \in \Omega, \ \ t>0.
\end{equation}
Now we mimic the proof of Theorem \ref{thm-upperb} and write $u = w\, f, \ v= w\, g$ with $w$ as in \eqref{eq-w} and 
$$
f,g \in H^1_0(\Omega, w^2 dx) = \left\{ f\in H^1_0(\Omega)\, :\, w\, (|\nabla f|+|f|) \in L^2(\Omega)\right\}. 
$$
Let $\widehat Q_\infty$ be the sesquilinear form on $H^1_0(\Omega, w^2 dx) \times H^1_0(\Omega, w^2 dx)$ defined by 
$$
\widehat Q_\infty[f,g] =  \int_\Omega \nabla \overline{f} \cdot \nabla g\ w^2 dx. 
$$ 
Now the boundary terms in \eqref{radial-der} vanish and for all positive functions $f,g\in H^1_0(\Omega, w^2 dx)$ we obtain the lower bound
\begin{align} \label{forms-equal}
Q_\infty[w f, w g] - \lambda \int_\Omega U\,  f g\, w^2 dx\, & \ \leq Q_\infty[w f, w g] - \lambda \int_\Omega\frac{1}{4 |x|^2} \left( \log \frac{|x|}{\rho} +\beta \right)^{-2}  f g\, w^2 dx \nonumber 
\\
&\  =  \widehat Q_\infty[f,g] .
\end{align}
By \cite[Thm.~2.24]{ouh} and \eqref{eq-kernels} this gives
\begin{equation} \label{hk-equal}
\e^{-t H_D(\lambda, U)}(x,y) \ \geq\  w(x)\, w(y) \, \e^{-t \widehat A_\infty}(x,y) \qquad \forall\ x,y\in M,
\end{equation}
where $\widehat A_\infty$ is the operator in $L^2(\Omega, w^2 dx)$ associated with the form $\widehat Q_\infty[\cdot\, , \cdot]$ with the form domain $H^1_0(M, w^2 dx)$. Note that since $\lambda < 1$, it follows from Lemma \ref{lem-aux} that 
$$
\int^\infty \frac{dt}{\V_2(x, \sqrt{t})} \ < \infty \qquad \forall\ x\in M.
$$
Hence the manifold $(M, w^2 dx)$ is non-parabolic. Since $\partial K$ is compact and $(M, w^2 dx)$ satisfies the parabolic Harnack inequality, see the proof of Theorem \ref{thm-upperb}, we may apply  \cite[Thm.~3.1]{gs2} with 
$(M,\mu)= (M, w^2 dx)$ and $\Omega$ as above. The latter says that there exists $\eps>0$ such that
$$
\e^{-t \widehat A_\infty}(x,y) \, \geq \,C'\,  \e^{-c'\, t \widehat A_\sigma}(x,y), \qquad t>0,\  x,y \in M\setminus K_\eps,
$$
holds for some $c',\,  C' >0$. Here $\widehat A_\sigma$ is the operator in $L^2(M,w^2 dx)$ defined in \eqref{A-hat-sigma}. From \eqref{ahat-two-sided} and Lemma \ref{lem-aux} we thus obtain 
\begin{equation} \label{ahat-lowerb}
\e^{-t \widehat A_\infty}(x,y) \, \geq \,  \frac{C'\, C_1\, \e^{-\frac{c\, |x-y|^2}{t}}}{\V_2(x, \sqrt{t})}\, \geq \,   \frac{C'\, C_1\, \e^{-\frac{c\, |x-y|^2}{t}}}{\pi\, t \left(\log\left(\frac{|x|+ \sqrt{ t} }{\rho} \right)+\beta \right)^{1+\sqrt{1-\lambda}}}\, .
\end{equation}
To complete the proof it suffices to apply \eqref{hk-equal} to the right hand side of \eqref{ahat-lowerb}. 
\end{proof}

\begin{remark} \label{rem-parabolic}
For $\lambda= 1$ the weighted manifold $(M, w^2 dx)$ becomes parabolic, see \eqref{eq-w} and \eqref{a-b}. 
Hence the Dirichlet heat kernel $\e^{-t \widehat A_\infty}(x,y)$ in this case has a faster decay in $t$ than the upper bound \eqref{hk-upperb-2d}. Indeed, by \cite[Thm.~4.9]{gs2} 
$$
\e^{-t \widehat A_\infty}(x,x) \ \asymp\  \frac{1}{t\, \log t \, \log (\log t)} \qquad t\to\infty
$$
holds for all $x$ far enough from $K$. This forbids an extension of Proposition \ref{prop-lowerb} to the case $\lambda = 1$. 
\end{remark}

\begin{remark}
A slightly more general notion of Gaussian estimates for a semigroup with kernel $k(t,x,y)$ consists in the inequality 
\[
|k(t,x,y)|\le C t^{-\frac{d}{2}} e^{-\frac{|x-y|^2}{bt}} e^{\omega t}
\]
for some constants $b,c>0$ and $\omega\in \mathbb R$ and all $t>0$ as well as almost every $x,y$. The advantage of this formulation is that also semigroups with complex-valued kernels can be discussed. Now, it is well-known that the semigroup generated by a (formal) Schrödinger operator $\Delta-V$ with potential $V\in L^p_{\rm loc}$ such that ${\rm Re }V\ge 0$ admits a modulus semigroup (i.e., a minimal dominating semigroup), which is then generated by $\Delta-{\rm Re }V$. This suggests a slight generalisation of Proposition~\ref{prop-lowerb}: the estimate~\eqref{hk-lowerb} accordingly holds if the kernel $e^{-tH_\sigma (\lambda,U)}(x,y)$ on the left hand side is replaced by its complex absolute value, provided $U$ is a complex-valued potential that satisfies ${\rm Re }U(x)\ge U_\beta (x)$ for some $\beta>0$ and all $x\in M$.
\end{remark}

\smallskip

%%%%%%%%%%%%%%%%%%%%%%%%%%%%%%%%%%%%%%%%%%%%%%%%%%%%%%%%%%%%
\subsection{A two-sided estimate} 

Here we provide a two-sided heat kernel estimate for $U=U_\sigma$.  

\begin{theorem} \label{cor-2-sided}
Let $K\subset\R^2$ be an open bounded and simply connected set with $C^{2}$ regular boundary. Let $0\le \lambda < 1$. 
Then there exist positive constants $C,c>0$ and $\eps>0$ such that for all $x,y\in M\setminus K_\eps$ and all $t>0$ we have
\begin{align} \label{eq-2-sided}
& \!\!\!\!\!\frac{ F_2(x,y;\lambda)\ \e^{-\frac{c\, |x-y|^2}{t}}}{C  t \left(\log\left(\frac{|x|+ \sqrt{t}}{\rho} \right)+ \frac{1}{\rho\, \sigma_0 }\right)^{1+\sqrt{1-\lambda}}}\,  \leq \, 
 \e^{-t H_\sigma(\lambda, \, U_\sigma)}(x,y)  
 \leq \frac{C\, F_2(x,y;\lambda)\ \e^{-\frac{ |x-y|^2}{ct}}}{ t \left(\log\left(\frac{|x|+ \sqrt{t}}{\rho} \right)+ \frac{1}{\rho\, \sigma_0 }\right)^{1+\sqrt{1-\lambda}}} , 
\end{align}
where 
$$
F_2(x,y;\lambda) := \left( \left(\log\frac{|x|}{\rho} + \frac{1}{\rho\, \sigma_0 }\right) \left(\log\frac{|y|}{\rho} + \frac{1}{\rho\, \sigma_0 }\right)\right)^{\frac{1+\sqrt{1-\lambda}}{2}} \ .
$$
\end{theorem}

\smallskip

\begin{proof}
The claim follows from Theorem \ref{thm-upperb} and Proposition \ref{prop-lowerb}. 
\end{proof}

\noindent Note that both upper and lower bound in \eqref{eq-2-sided} are decreasing functions of $\sigma_0$. 

\begin{remark}
For small times the diagonal element of the behavior of the heat kernel is not affected by the presence of the boundary, neither by the potential $U_\sigma$. In fact
$$
\e^{-t H_\sigma(\lambda, \, U_\sigma)}(x,x) \ \asymp \ t^{-1} \qquad t\to 0.
$$
On the other hand, for large times we hav 
$$
\e^{-t H_\sigma(\lambda, \, U_\sigma)}(x,x) \ \asymp \ t^{-1}\, ( \log t)^{-1-\sqrt{1-\lambda}}  \qquad t\to \infty.
$$
\end{remark}

\begin{remark}
Consider the case \underline{$\sigma=\, $const} and \underline{$\lambda=0$}. To simplify the notation we write $H_\sigma$ instead of $H_\sigma(0, U_\sigma)$. By domination of semigroups we then have 
\begin{equation} 
\e^{-t H_D }(x,x) \, \leq \, \e^{-t H_\Sigma}(x,x)\, \leq \, \e^{-t H_{\sigma}}(x,x)\, \leq \, \e^{-t H_0}(x,x)\, \qquad x\in M\setminus K_\eps \ ,
\end{equation}
for all $0< \sigma < \Sigma$. By passing to the limit $\sigma\to 0$ in \eqref{eq-2-sided},  for a fixed $x \in M\setminus K_\eps$, we thus obtain 
\begin{equation}  \label{neumann-bc}
\e^{-t H_0}(x,x)  \  \asymp \ t^{-1}  \hspace{1.5 cm} \text{(Neumann boundary conditions)}
\end{equation}
\end{remark}

\smallskip

\begin{remark}
There are two reasons why Theorem \ref{cor-2-sided}  is not completely satisfactory. First, the lower bound is non-zero only for $x$ far enough from $K$. This is because we use the Dirichlet heat kernel as a bound from below, see \eqref{lowerb-1} and \eqref{hk-equal}. Second, it does not cover the critical case $\lambda = 1$, see Remark \ref{rem-parabolic} for details.
Both these artifacts can be removed in the special case when $K=B(0, \rho)$ and $\sigma$ is constant. This suggests that the assertion of Theorem~\ref{cor-2-sided} might actually be improved.
\end{remark}

\subsubsection*{\bf Example: a ball with constant $\sigma$} 

\begin{proposition} \label{prop-ball}
Let $K=B(0,\rho),  \sigma=\sigma_0$ and let $U$ be as in Theorem \ref{cor-2-sided}. Then the two-sided estimate \eqref{eq-2-sided} holds for all $x,y\in M$ and all $\lambda\in [0, 1]$.
\end{proposition}

\begin{proof}
Take $w$ as in \eqref{eq-w} with $\alpha$ and $\beta$ given by \eqref{a-b}. We then obtain the identity 
\begin{equation} \label{form-ball}
Q_\sigma[w f, w g] -  \lambda \int_M U \overline{f} g\, w^2 dx\, =\,  \widehat Q_\sigma[f,g],  
\end{equation}
where $\widehat Q_\sigma[ \cdot\, , \cdot ]$ is defined in \eqref{form-dominance}. The above equation holds for all functions $f,g \in H^1(M, w^2 dx)$. Hence 
$$
 \e^{-t H_\sigma(\lambda, U)}(x,y) = w(x)\, w(y)\, \e^{-t \widehat A_\sigma}(x,y),\qquad \forall\ x,y\in M, \quad t>0.
$$
The claim now follows from \eqref{ahat-two-sided} and Lemma \ref{lem-aux}. 
\end{proof}

\smallskip

%%%%%%%%%%%%%%%%%%%%%%%%%%%%%%%%%%%%%%%%%%%%%%%%%%%%%%%%
\section{\bf Applications}
\label{sec-appl}
\noindent In this section we will apply the heat kernel bounds obtained in section \ref{sec-upperb} to establish spectral estimates for 
two-dimensional Schrödinger operators on exterior domains. We begin with a simple but important consequence of the proof of Theorem \ref{thm-upperb}. 
%%%%%%%%%%%%%%%%%%%%%%%%%%%%%%%%%%%%%%%%
\subsection{A Hardy inequality}

\begin{proposition} \label{prop-hardy}
In addition to the Assumption~\ref{ass-U}, let $K\subset\R^2$ have $C^{2}$-regular boundary. Then for all $u\in H^1(M)$ it holds
\begin{equation} \label{eq-hardy}
Q_\sigma[u,u]  \, \geq\, \frac 14 \int_M  \frac{|u(x)|^2}{|x|^2 \left( \log \frac{|x|}{\rho} +\frac{1}{2 \rho\,  \sigma_0 }\right)^{2}}\, dx,
\end{equation}
where $Q_\sigma[\cdot\, , \cdot]$ is given by \eqref{q-form}.  Moreover, the above inequality fails if we replace the constant $\frac 14$ on the right hand side by any constant $C> \frac 14$.  
\end{proposition}

\begin{proof}
As above we write 
\begin{equation} \label{fact-2}
u(x) =  \left( \log \frac{|x|}{\rho} +\frac{1}{2 \rho  \sigma_0 }\right)^{\frac 12}\, f(x)\, .
\end{equation}
Inequality \eqref{eq-hardy} then follows immediately from \eqref{form-dominance} applied with $\lambda=1$. To prove the sharpness of the constant $1/4$ we consider the example $K=B(0,1)$ and $\sigma=\sigma_0>0$ treated already in Proposition \ref{prop-ball}. Using the factorization \eqref{fact-2} with a radial test function $u$ we then obtain the following identity;
\begin{align} \label{sharp}
Q_\sigma[u,u] & - C\int_M  \frac{|u(x)|^2}{|x|^2 \left( \log \frac{|x|}{\rho} +\frac{1}{2 \rho\,  \sigma_0 }\right)^{2}}\ dx  =
2\pi \int_1^\infty (f'(r))^2\ r\,  \left( \log r +\frac{1}{2 \rho\,  \sigma_0 }\right)\, dr  \nonumber\\
&\qquad\qquad \qquad\qquad -2\pi \left(C-\frac 14\right) \int_1^\infty f^2(r)\, r^{-1}\, \left( \log r +\frac{1}{2 \rho\,  \sigma_0 }\right)^{-1}\, dr. 
\end{align}
If we now set
$$
f(r)= f_n(r) = 
\left\{
\begin{array}{l@{\qquad}l}
\log \left(\, \log \left(1-\log \frac{r}{n}\right) +1\right) & \mathrm{if}\  r\leq n \, , \\
  &  \\
0  &  \mathrm{if}\ n < r\, 
 \end{array}
\right.
 \qquad  n\in\N,
$$ 
then $f_n\in H^1(M)$ and a direct calculation shows that the right hand side of \eqref{sharp} is negative for $n$ large enough whenever $C > 1/4$. 
\end{proof}

\begin{remark} 
Hardy-type inequalities for Laplace operators with Robin boundary conditions were recently established in \cite{kl}.  Among other things it was  shown in \cite[Thm.~5.1]{kl} that for a constant $\sigma$ the inequality
\begin{equation} \label{hardy-kl}
Q_\sigma[u,u] \, \geq\,  \frac 14 \int_{M} \left(\Big (|x|-\rho+\frac{1}{2\sigma}\Big)^{-2} +\frac{(d-1)(d-3)}{|x|^2}\right) |u(x)|^2\, dx
\end{equation}
holds for all $u\in H^1(M)$, where $M=\R^d\setminus B(0,\rho)$. 
Note however, for $d=2$ the integral weight on the right hand side of \eqref{hardy-kl} is positive only for $\sigma$ large enough. Moreover, still for $d=2$, this weight decays as $|x|^{-3}$ for $|x| \to\infty$, whereas the integral weight in \eqref{eq-hardy} has the optimal decay rate $|x|^{-2}\, (\log |x|)^{-2}$.
\end{remark}

\begin{corollary} \label{cor-hardy}
In addition to the Assumption~\ref{ass-U}, let $K\subset\R^2$ have $C^{2}$-regular boundary. Then for all $u\in H_0^1(M)$ it holds
\begin{equation} \label{eq-hardy-dir}
\int_M |\nabla u|^2\, dx  \, \geq\, \frac 14 \int_M  \frac{|u(x)|^2}{|x|^2}\  \left( \log \frac{|x|}{\rho} \right)^{-2}\ dx,
\end{equation}
with the sharp constant $1/4$.
\end{corollary}

\begin{proof}
In view of the monotone convergence theorem the claim follows by applying \eqref{eq-hardy} to $u\in H_0^1(M)$ and letting $\sigma_0\to \infty$. 
\end{proof}

%%%%%%%%%%%%%%%%%%%%%%%%%%%%%%%%%%%%%%%%
\subsection{Hardy-Lieb-Thirring inequalities}  

It is well known that the Laplace operator satisfies, in the sense of quadratic forms on $H^1(\R^d)$, the Hardy inequality
\begin{equation}
-\Delta \, \geq \, \frac{(d-2)^2}{4 |x|^2}\qquad d\geq 3,
\end{equation}
with the sharp constant $(d-2)^2/4$. Motivated by this fact Ekholm and Frank established in \cite{ef} the so-called Hardy-Lieb-Thirring inequalities, i.e.~estimates for the moments of negative eigenvalues $\{-\lambda_j(V)\}$ of a Schr\"odinger operator $-\Delta-\frac{(d-2)^2}{4 |x|^2}-V$ in terms of a suitable $L^p-$ norm of $V$. More precisely, they proved that 
\begin{equation} \label{hlt-ef}
\tr \left(-\Delta-\frac{(d-2)^2}{4 |x|^2}-V \right)_-^\gamma = \sum_j \lambda_j(V)^\gamma \ \leq\ C_{d,\gamma} \int_{\R^d} V(x)_+^{\gamma+\frac d2}\, dx \qquad \hbox{if }d\geq 3
\end{equation}
holds true for all $\gamma>0$ and  some constant $C_{d,\gamma}$ independent of $V$, see also \cite{fr}. This improves considerably the classical Lieb-Thirring estimates,  \cite{LT}, by the presence of the negative factor $-\frac{(d-2)^2}{4 |x|^2}$ on the left hand side. 

\medskip

\noindent Now, our Corollary \ref{cor-hardy} shows that the inequality
$$
-\Delta_D \, \geq \, U_\infty 
$$
holds in the sense of quadratic forms on $H_0^1(\R^2\setminus K)$.  Since the constant $\frac 14$ is sharp, it is natural to ask whether an analog of \eqref{hlt-ef} holds for the operator 
$$
H_D\left(1 , U_\infty\right) = -\Delta_D- U_\infty
$$
in $L^2(M)$ with Dirichlet boundary conditions. With the help of Theorem \ref{cor-dirichlet} we obtain 

\begin{theorem} \label{thm-hlt}
Let  $R_{\rm in}(K)=\rho$. For every $\gamma>0$ there exists $C(\gamma,\rho)$ such that 
\begin{equation} \label{hlt}
\tr \left(-\Delta_D -U_\infty -V \right)_-^\gamma \ \leq\ C(\gamma,\rho)  \int_{M} V_+(x)^{\gamma+1}\, dx 
\end{equation}
holds true for all $V\in L^{\gamma+1}(M)$. 
\end{theorem}

\begin{proof}
By the min-max principle it suffices to prove \eqref{hlt} for $V\geq 0$. The inequality of Lieb, see \cite{lieb}, yields the upper bound
\begin{equation} \label{eq-lieb}
\tr\left(-\Delta_D -U_\infty  -V\right)_-^\gamma \, \leq \,
L_{b,\gamma}\, \int_M\, \int_{0}^\infty  \, \e^{-t H_D\left(1, U_\infty\right)}(x,x)\, t^{-1-\gamma}\, (t\, V(x)-b)_+\, dt\, dx,
\end{equation}
where $b>0$ is arbitrary and 
\begin{equation} \label{Lb}
L_{b,\gamma} = \Gamma(\gamma+1)\, \left(
e^{-b}-b\, \int_b^\infty\, s^{-1}\, e^{-s}\, ds\right)^{-1}.
\end{equation}
On the other hand, Theorem \ref{cor-dirichlet} implies that 
$$
\e^{-t H_D\left(1, U_\infty\right)}(x,x) \ \leq C\ \frac{\log\frac{|x|}{\rho} }{t \, \log \frac{|x| +\sqrt{t}}{\rho}}\  \leq 
\ \frac{C}{t} \qquad \forall\ x \in M, \quad t>0,
$$
with some constant $C$ independent of $x$. Inequality \eqref{hlt} now follows by inserting the above upper bound in \eqref{eq-lieb} and integrating with respect to $t$. 
\end{proof}

\noindent In the sequel we denote by 
$$
N(-\Delta_D -\lambda\, U_\infty  -V,0) := \tr\left(-\Delta_D -U_\infty  -a V\right)_-^0 
$$
the number of negative eigenvalues, counted with their multiplicities, of the operator $-\Delta_D -\lambda\, U_\infty  -V$. 
 
\begin{remark}
Inequality \eqref{hlt}, similarly as \eqref{hlt-ef}, fails when $\gamma=0$. Indeed, if for $K=B(0,\rho)$, then a standard test function argument shows that the operator  $-\Delta_D -U_\infty -a V,$
with some  $V\geq 0, \, V\neq 0,$ has at least one negative eigenvalue for any $a>0$. Hence
$$
 N(-\Delta_D -U_\infty  -a V,0)  \, \geq \, 1 \qquad \forall\ a>0,
$$
which contradicts \eqref{hlt} for $\gamma=0$ and $a$ small enough.
\end{remark}

\noindent If $\lambda <1$, then the operator  $H_D\left(\lambda, U_\infty\right)$ is sub-critical, and it is possible to extend inequality \eqref{hlt} to the border-line case $\gamma=0$. 

\begin{theorem} \label{thm-hclr}
Assume that $R_{\rm in}(K)= \rho=1$. There exists a constant $C_0$ such that 
\begin{equation} \label{hclr}
N(-\Delta_D -\lambda\, U_\infty  -V,0) \ \leq\ \frac{C_0}{\sqrt{1-\lambda}}\,  \int_{M} V_+(x) \left(\log V_+(x)\right)^{-\sqrt{1-\lambda}}\,  (\log |x|)^{1+\sqrt{1-\lambda}}\, dx 
\end{equation}
holds true for all $0\leq \lambda<1$ and for all $V$ for which the right hand side is finite. 
\end{theorem}

\begin{proof}
As in the proof of Theorem \ref{thm-hlt} we assume without loss of generality that $V\geq 0$.  
Since $|x| \geq 1$, Theorem \ref{cor-dirichlet} implies that the upper bound 
$$
\e^{-t H_D\left(\lambda, U_\infty\right)}(x,x) \ \leq \frac{C\, \left(\log |x| \right)^{1+\sqrt{1-\lambda}}}{t \left(\log (|x| +\sqrt{t}) \right)^{1+\sqrt{1-\lambda}}}\  \leq 
\ \frac{4\, C\, \left(\log |x| \right)^{1+\sqrt{1-\lambda}} }{t\, \left(\log t \right)^{1+\sqrt{1-\lambda}}} 
$$
holds for all $x\in M$ and $t>0$. Hence in view of \eqref{eq-lieb} we have 
\begin{align*}
N(-\Delta_D -\lambda\, U_\infty  -V,0) &\  \leq\  L_{b,0}  \int_M\, \int_{0}^\infty  \, \e^{-t H_D\left(\lambda, U_\infty\right)}(x,x)\, t^{-1}\, (t\, V(x)-b)_+\, dt\, dx \\
& \ \leq \   C_0 \int_M \left(\log |x| \right)^{1+\sqrt{1-\lambda}} \, \int_{0}^\infty \frac{(t\, V(x)-b)_+}{t^2\, (\log t)^{1+\sqrt{1-\lambda}}}\ dt\, dx,
\end{align*}
where $C_0 = 4\, C \,  L_{b,0}$. 
The claim follows by  and choosing $b=1$ and integrating with respect to $t$.
\end{proof}

\begin{remark}
As expected the constant on the right hand side of \eqref{hclr} diverges as $\lambda$ approaches the critical value $1$. 
The presence of a logarithmic weight in the estimates for the number of negative eigenvalues of Schr\"odinger operators is typical for the two-dimensional case, see e.g.~\cite{ls,sol1,st}. 
\end{remark}

%%%%%%%%%%%%%%%%%%%%%%%%%%%%%%%%%%%%%%%%%%%%%
\section{\bf The case $d=1$}
\label{sec-1d}
\noindent 
The effect of boundary conditions on the large behavior of the heat kernel in dimension one is even more robust than the case $d=2$. In order to see why let us consider the half-line $M=\R_+$ and the Laplacian with Robin boundary condition at zero associated with the symmetric sesquilinear form
\begin{equation}
\Q_\sigma [u,v] = \int_0^\infty \overline{u'} v'\, dx +\sigma\, \overline{u(0)}\, v(0) , \qquad u,v \in H^1(\R_+), \quad \sigma>0\, .
\end{equation}
For definiteness we will assume that the potential $U$ is given by
$$
U(x) := \U_\sigma(x) : =\frac14\, \left(x+\frac{1}{\sigma}\right)^{-2}\, .
$$
By mimicking the arguments of section \ref{sec-prelim} it is easy to verify that the form 
\begin{equation} \label{from-1d}
\Q_\sigma [u,v] -\lambda \int_0^\infty U\, \overline{u}\, v\, dx, \qquad u,v\in H^1(\R_+)\, .
\end{equation}
is closed for all $\lambda\in\R$ and that the associated operator, which we denote by $\HH_\sigma(\lambda,U)$, generates in $L^2(\R_+)$
an ultracontractive semigroup with integral kernel 
$$
\e^{-t\, \HH_\sigma(\lambda, \, \U_\sigma )}(x,y), \qquad x,y\in M, \quad t>0. 
$$
We have 

\begin{theorem} \label{thm-1d}
Let $0\leq \lambda \leq 1$. Then there exit positive constants $C,c>0$ such that for all $x,y\in\R_+$ and all $t>0$ 
\begin{align*} %\label{hk-1d}
\frac{C \, F_1(x,y;\lambda)\  \e^{-\frac{c\, |x-y|^2}{t}}}{\sqrt{t}\, \left( x+\sqrt{t}+\frac{1}{\sigma} \right)^{1+\sqrt{1-\lambda}}}  & \ \leq \ \e^{-t\, \HH_\sigma(\lambda, \, \U_\sigma )}(x,y) 
\ \leq\   \frac{ F_1(x,y;\lambda)\  \e^{-\frac{ |x-y|^2}{c t}}}{C\sqrt{t}\, \left( x+\sqrt{t}+\frac{1}{\sigma} \right)^{1+\sqrt{1-\lambda}}}  \, ,
\end{align*}
where 
$$
F_1(x,y;\lambda) = \left(\left( x+\frac{1}{\sigma} \right)\left( y+\frac{1}{\sigma} \right)\right)^{\frac{1+\sqrt{1-\lambda}}{2}}\, .
$$
\end{theorem}

\smallskip

\begin{proof}
We will proceed in a similar way as in the proof of Theorem \ref{thm-upperb}.  Here we choose the weight function in the form
\begin{equation} \label{omega}
\omega(x) := \left(x+\frac \alpha\sigma\right)^\alpha, 
\end{equation}
with $\alpha$ as in \eqref{a-b}, and substitute $u=\omega\, f,\ v=\omega\, g$ with $f,g \in H^1(\R_+, \omega^2 dx)$. A direct calculation then yields
\begin{align*}
\Q_\sigma [u,v] -\lambda \int_0^\infty \U_\sigma\, \overline{u}\, v\, dx & = \int_0^\infty \overline{f'}\, g'\, \omega^2\, dx.
\end{align*}
Hence 
\begin{equation} \label{trans-1d}
\e^{-t \, \HH_\sigma(\lambda, \, \U_\sigma )}(x,y) =\omega(x)\, \omega(y)\ \e^{-t\, \A_\sigma }(x,y), \qquad x,y\in\R_+,
\end{equation}
where $\A_\sigma$ is the self-adjoint operator in $L^2(\R_+,\, \omega^2 dx)$ associated with the quadratic form 
$$
 \int_0^\infty |f'|^2\, \omega^2\, dx
$$
with the form domain $H^1(\R_+, \omega^2 dx)$. Since
\begin{equation}
\sup_{0\leq x\leq 2r}\, \omega(x) \ \leq \ C' \inf_{r \leq x}\, \omega(x), \qquad \forall\ r >0,
\end{equation} 
with $C'$ independent of $r$, it follows from \cite[Thm.~2.11]{gs2} that the weighted manifold $(\R_+, \omega^2 dx)$ satisfies the parabolic Harnack inequality. 
By \cite[Thm.~2.8]{gs2} we thus infer that there exist positive constants $C$ and $c$ such that
$$
\frac{C\, \e^{-\frac{c\, |x-y|^2}{t}}}{\V_1(x,\sqrt{t}) }\ \leq \ \e^{-t\, \A_\sigma }(x,y)\ \leq \ \frac{\e^{-\frac{ |x-y|^2}{ct}}}{C\, \V_1(x,\sqrt{t}) }\, ,
$$
holds for all $x,y\in \R_+$ and $t>0$, where 
$$
\V_1(x,\sqrt{t}) = \int_{x-\sqrt{t}}^{x+\sqrt{t}} \omega^2(y)\, dy.
$$
Since 
$$
\frac{\sqrt{t}}{2^{2\alpha+1}}\, \left( x+\sqrt{t}+\frac{\alpha}{\sigma} \right)^{2\alpha}\, \leq \,  \V_1(x,\sqrt{t}) \, \leq\, 2\, \sqrt{t}\, \left( x+\sqrt{t}+\frac{\alpha}{\sigma} \right)^{2\alpha}, 
$$
the claim follows from \eqref{trans-1d}. 
\end{proof}

%%%%%%%%%%%%%%%%%%%%%%%%%%%%%%%%%%%%%%%%%%%%%
\section{\bf The case $d \geq 3$}
\label{sec-hd}
\noindent Contrary to the cases $d=1$ and $d=2$, in higher dimensions the presence of Robin boundary conditions on $\partial M$ does not accelerate the decay of the associated heat kernel, at least as long as no potential is introduced. Indeed, using the domination of semigroups and \cite[Thm.~3.1]{gs2} we obtain 

\begin{proposition}  \label{prop-hd}
Let $d\geq 3$ and let $K\subset\R^d$ be an open bounded and simply connected set with Lipschitz boundary. Let $M=\R^d\setminus K$. Assume that $\sigma \in L^\infty(\partial M)$. Then there exist positive constants $\eps$, $c$ and $C$ such that
\begin{equation}  \label{lowerb-hd}
C^{-1}\,  t^{-\frac d2}\ \e^{-\frac{c |x-y|^2}{t}} \ \leq \  \e^{t \Delta_\sigma}(x,y) \ \leq \ C\,  t^{-\frac d2}\ \e^{-\frac{|x-y|^2}{ct}}
\end{equation}
holds for all $t>0$ and all $x,y\in M\setminus K_\eps$.  
\end{proposition}

\begin{proof} 
By domination of semigroups we have 
\begin{equation} \label{dn-domination} 
\e^{t \Delta_D}(x,y) \ \leq \  \e^{t \Delta_\sigma}(x,y) \ \leq \ \e^{t \Delta_0}(x,y), \quad x,y\in M, \ t>0.
\end{equation}
Since $\R^d, \, d\geq 3,$ equipped with the Euclidean metric is a non-parabolic manifold which satisfies the parabolic Harnack inequality, we can apply \cite[Thm.~3.1]{gs2}. The latter implies that  
$$
\e^{t \Delta_D}(x,y)\, \geq C^{-1}\,  t^{-\frac d2}\ \e^{-\frac{c |x-y|^2}{t}}\, .
$$
This proves the lower bound in \eqref{lowerb-hd}. The upper bound follows from \eqref{neumann-ub} and \eqref{dn-domination}.  
\end{proof}

\medskip

%%%%%%%%%%%%%%%%%%%%%%%%%%%%%%%%%%%%%%%%%%%%%%%%%%%%%%%%%%%%%%%%%%%%
%%%%%%%%%%%%

\appendix
\section{}\label{sec:app} 

\begin{lemma} \label{lem-harnack}
In addition to Assumption~\ref{ass-U}, let $K\subset\R^2$ have a $C^{2}$-regular boundary. Let $w$ be given by \eqref{eq-w} with $\alpha\in [1/2, 1]$ and $\beta>0$. Then there exists a constant $c$, independent of $\alpha$ and $\beta$,  such that 
\begin{equation}
\sup_{x\in B(x_0, 2r)}\, w(x) \ \leq \ c \inf_{x \in M\setminus B(x_0, r)}\, w(x) 
\end{equation} 
holds for any  $x_0\in M$ and all $r$ large enough. 
\end{lemma} 

\begin{proof}
Take $\gamma >2$ large enough so that $\overline{K} \subset B(x_0, \gamma\, |x_0|)$. Then for any $r\geq \gamma\, |x_0|$ we have 
\begin{equation} \label{sup}
\sup_{x\in B(x_0, 2r)}\, w(x) \ \leq \ \left(\log\frac{|x_0|+2r}{\rho}+\beta\right)^{\alpha} \, \leq\, \left(\log\frac{r\, (2+\gamma^{-1})}{\rho}+\beta\right)^{\alpha}\, .
\end{equation} 
On the other hand, since $r\geq\, \gamma\, |x_0| \geq \gamma\, \rho$ for $c>1$ it holds
\begin{align*}
c \inf_{x \in M\setminus B(x_0, r)}\, w(x)  & \geq \left(c \log\frac{r-|x_0|}{\rho}+\beta\right)^{\alpha}\, \geq \left(c \log\frac{r\, (1-\gamma^{-1})}{\rho}+\beta\right)^{\alpha} \\
& \geq \left( \log\frac r\rho +\log\, ( \gamma^{c-1}\, (1-\gamma^{-1})^c) +\beta\right)^{\alpha}\, ,
\end{align*} 
where we have used the fact that $1/2 \leq \alpha \leq 1$, see \eqref{a-b}. In view of \eqref{sup} it thus suffices to take $c$ large enough such that 
$$
\gamma^{c-1}\, (1-\gamma^{-1})^c \ \geq\ 2+\gamma^{-1}\, .
$$
\end{proof}

\begin{lemma} \label{lem-aux}
Let $w$ be given by \eqref{eq-w} with $\alpha, \beta>0$. Then there exists a constant $c_0>0$, independent of $\beta$, such that 
\begin{equation} \label{V-lowerb}
c_0\ t \, \left(\log\left(\frac{|x|+\sqrt{t}}{\rho}\right)+\beta\right)^{2\alpha} \ \leq \ 
\V_2(x,\sqrt{t} )\ \leq \ \pi\, t \, \left(\log\left(\frac{|x|+\sqrt{t}}{\rho}\right)+\beta\right)^{2\alpha} 
\end{equation}
holds for all $x\in M$ and $t>0$, where $\V_2(x,\sqrt{t} )$ is defined in \eqref{volume}. 
\end{lemma}

\begin{proof}
The upper bound is obvious. To establish the lower bound we first prove the following auxiliary estimate:
\begin{equation} \label{eq-aux-2}
\forall \ \delta\in(0,1)\, :  \  \log\left(\frac{|x|+\delta s}{\rho}\right) \, \geq\,  \delta\, \log\left(\frac{|x|+ s}{\rho}\right) \qquad \forall \ x\in M, \ \forall\ s>0. 
\end{equation}
Indeed, since $|x|\geq \rho$ for all $x\in M$, by the Bernoulli inequality we have 
\begin{align*} 
\left(\frac{|x|+\delta s}{\rho}\right)^{\frac 1\delta} & = \left(\frac{|x|}{\rho}\right)^{\frac 1\delta}\, \left(1+\frac{\delta s}{|x|}\right)^{\frac 1\delta} \, \geq\, \frac{|x|}{\rho}\, \left(1+\frac{s}{|x|}\right) = \frac{|x|+s}{\rho}\, ,
\end{align*} 
which implies \eqref{eq-aux-2}. Next, since $K$ is bounded there exists $k_0>0$ such that
\begin{equation} \label{aux-1}
K \subset B\left(0, \frac{k_0}{2}\right) .
\end{equation}
Consider first the case when $|x|+\sqrt{t} > k_0$.  
Then there exists a circular segment $\B(x, \sqrt{t})$ of $B(x, \sqrt{t})$ with height $\frac{\sqrt{t}}{2}$ and such that 
\begin{equation} \label{aux-2}
\B(x, \sqrt{t}) \subset \left(B(x, \sqrt{t}) \setminus B\Big(0, |x|+\frac{\sqrt{t}}{2}\, \Big) \right) \subset M.
\end{equation}
The latter implies that
$$
w(y) \geq  \left(\log\left(\frac{|x|}{\rho} +\frac{\sqrt{t}}{2\rho} \right)+\beta\right)^{\alpha} \qquad \forall\ y \in \B(x, \sqrt{t}). 
$$
An elementary calculation now shows that 
\begin{align} \label{big-ball}
\V_2(x,\sqrt{t} )  & \geq \int_{\B(x, \sqrt{t})} \ w^2(y)\, dy  \, \geq \, \inf_{y\in \B(x, \sqrt{t})} w^2(y)\, \int_{\B(x, \sqrt{t})} \, dy \nonumber \\
& = t \left(\frac{\pi}{3}-\frac{\sqrt{3}}{4}\right)\,  \inf_{y\in \B(x, \sqrt{t})} w^2(y)\nonumber \\
& \geq t \left(\frac{\pi}{3}-\frac{\sqrt{3}}{4}\right)\, \left(\log\left(\frac{|x|}{\rho} +\frac{\sqrt{t}}{2\rho} \right)+\beta\right)^{2\alpha}\nonumber \\
& \geq  t \left(\frac{\pi}{3}-\frac{\sqrt{3}}{4}\right)\, 2^{-2\alpha}\, \left(\log\left(\frac{|x|+\sqrt{t}}{\rho} \right)+\beta\right)^{2\alpha}\, ,
\end{align}
where we have used \eqref{eq-aux-2}. 

 Assume now that $|x| + \sqrt{t} \leq k_0$. In this case we pick $\eps>0$ and distinguish two separate situations:

\smallskip

 \underline{1. $x\in M\setminus B(0, \rho+\eps)$}. For any fixed $x$ we consider the function 
$$
F_x(t) = t^{-1}\,  \int_{B(x, \sqrt{t})\cap M} \ \left(\log \frac{|y|}{\rho}\right)^{2\alpha}\, dy
$$
as a function of $t$ on the interval $(0, (k_0-|x|)^2]$. This function is obviously positive and continuous. On the other hand, $K$ satisfies exterior cone condition. By the Lebesgue property we thus have
$$
\liminf_{t\to 0} F_x(t) \, \geq\,  C\, \left(\log \frac{|x|}{\rho}\right)^{2\alpha}\, \geq \, C\, \left(\log \frac{\rho+\eps}{\rho}\right)^{2\alpha} \, ,
$$
with some constant $C>0$ independent of $x$. Hence 
$$
\inf_{0< t\leq (k_0-|x|)^2} \, F_x(t) \, \geq \, c_\eps\, \log\left(\frac{k_0}{\rho}\right)^{2\alpha}\, 
$$
for some $c_\eps>0$. It follows that 
\begin{equation} \label{lowerb-2}
\int_{B(x, \sqrt{t})\cap M} \ \left(\log \frac{|y|}{\rho}\right)^{2\alpha}\, dy \, \geq\, c_\eps\ t\, \log\left(\frac{k_0}{\rho}\right)^{2\alpha}\, \geq c_\eps\ t\, \log\left(\frac{|x|+\sqrt{t}}{\rho}\right)^{2\alpha}
\end{equation} 
holds for all $x\in M\setminus B(0, \rho+\eps)$ and all $t\in (0, (k_0-|x|)^2]$. 

\smallskip

 \underline{2. $x\in M\cap B(0, \rho+\eps)$}. From the fact that the curvature of $\partial K$ is bounded, by assumption, it follows that by taking  $\eps\leq \eps_0$, with $\eps_0$ small enough, we can ensure that there exists  $\gamma\in (0,\pi)$ and a constant $\delta_\gamma \in (0,1)$ such for any $x\in M\cap B(0, \rho+\eps)$ and $t$ small enough the ball $B(x, \sqrt{t})$ contains a circular section $B_\gamma(x, \sqrt{t})\subset M$ with the opening angle $\gamma$ which satisfies
$$
\forall\ y\in B_\gamma(x, \sqrt{t})\, : \ |y| \geq |x| + \delta_\gamma \, |x-y| .
$$
We then have 
\begin{align*}
 \int_{B(x, \sqrt{t})\cap M} \ \left(\log \frac{|y|}{\rho}\right)^{2\alpha}\, dy & \geq  \int_{B_\gamma(x, \sqrt{t})} \ \left(\log \frac{|y|}{\rho}\right)^{2\alpha}\, dy  \geq \frac{\gamma}{2}\, \int_0^{\sqrt{t}} \left(\log \frac{|x|+\delta_\gamma\, s}{\rho}\right)^{2\alpha}\, s\,  ds \\
 &
 \geq \frac{\gamma}{2}\, \int_{\sqrt{t}/2}^{\sqrt{t}} \left(\log \frac{|x|+\delta_\gamma\, s}{\rho}\right)^{2\alpha}\, s\,  ds\\  
 & \geq 
\frac{\gamma}{2}\,  \left(\log \frac{|x|+\frac{\delta_\gamma\, \sqrt{t}}{2}}{\rho}\right)^{2\alpha} \int_{\sqrt{t}/2}^{\sqrt{t}}\, s\,  ds  \geq C' \, t\, \log\left(\frac{|x|+\sqrt{t}}{\rho}\right)^{2\alpha}\, , 
\end{align*}
where we have used \eqref{eq-aux-2}. The constant $C'>0$ here depends only on $\eps_0$ and $\gamma$. Using the same reasoning as above we thus conclude that estimate \eqref{lowerb-2} holds, with a different constant, also for all $x\in M\cap B(0, \rho+\eps)$ and all $t\in (0, (k_0-|x|)^2]$.  Since $\frac 12 \leq \alpha\leq 1$, see equation \eqref{a-b}, this in combination with \eqref{big-ball} shows that
$$
 \V_2(x,\sqrt{t} ) \, \geq \, c_0\ t\, \log\left(\frac{|x|+\sqrt{t}}{\rho}+\beta\right)^{2\alpha}\, 
$$
for all $x\in M$ and  some $c_0$. 
\end{proof}

\noindent As a consequence of Lemma \ref{lem-aux} we obtain the volume doubling property of the manifold $(M, w^2 dx)$.

\begin{lemma} \label{lem-vd}
Let $w$ be given by \eqref{eq-w} with $\alpha, \beta>0$. Then 
$$
\V_2(x, 2r) \ \leq \ \frac{4^{\alpha+1} \pi}{c_0}\ \V_2(x,r) , \qquad \forall\ x\in M, \ \forall \ r>0,
$$
where $c_0$ is given by Lemma \ref{lem-aux} and $\V_2(x,\sqrt{t} )$ is defined in \eqref{volume}. 
\end{lemma}

\begin{proof}
The claim follows by choosing $\delta= \frac 12$ in \eqref{eq-aux-2}. 
\end{proof}

\section{}\label{sec:sobolev} 
 Let  $w$ be given by \eqref{eq-w} with $\alpha, \beta >0$ and let $L^p(M,w^2\, dx), p>1,$ be the space of functions $f$ such that 
$$
\|f\|_{p,w}^p := \int_M |f(x)|^p\, w^2(x)\, dx < \infty. 
$$ 
We have 
\begin{lemma} \label{lem-sobolev} For any $q\in (2,\infty)$ there exists a constant $C_q$ such that 
\begin{equation} \label{eq-sob-w}
\|f\|_{q,w} \ \leq\ C_q \, \|f\|_{H^1(M,w^2\, dx)} \qquad \forall\, f\in H^1(M,w^2\, dx), 
\end{equation}
where
$$
 \|f\|_{H^1(M,w^2\, dx)} := \left( \int_M (|\nabla f|^2+|f|^2)\, w^2\, dx\right)^{\frac 12}\, .
$$
\end{lemma}

\begin{proof} 
Let $f\in H^1(M,w^2\, dx)$ and set $u: = f\, w^{\frac 2q}$. Note that in view of \eqref{eq-w} 
\begin{equation} \label{w-aux}
 \inf_{x\in M} w(x)  > 0\, , \qquad \sup_{x\in M} \frac{|\nabla w(x)|}{w(x)} \, < \infty\, .
\end{equation} 
Since $2q >4$ by assumption, equation \eqref{w-aux} implies that for some $c_0$ there holds
$$
\int_M |f|^2\, w^{\frac 4q}\, dx \, \leq\,  c_0\,  \|f\|_{2,w}^2\, , \qquad \int_M |\nabla f|^2\, w^{\frac 4q}\, dx\, \leq\, c_0\,  \|\nabla f\|_{2,w}^2\, .
$$
Hence by H\"older's inequality we get 
\begin{align*}
\int_M |\nabla u|^2\, dx & = \int_M \left( |\nabla f|^2\,  w^{\frac 4q}  + \frac 4q\, f w^{\frac 4q -1}\, \nabla f\cdot\nabla w + \frac{4}{q^2} |f|^2\, w^{\frac 4q-2}\, |\nabla w|^2 \right)\, dx \nonumber \\
& \leq c_0\,  \|\nabla f\|_{2,w}^2 + \frac{4 c_0}{q}\, \left\|\frac{\nabla w}{w}\right\|_{L^\infty(M)}\,   \| f\|_{2,w} \,  \|\nabla f\|_{2,w}  + \frac{4 c_0}{q^2}\, \left\|\frac{\nabla w}{w}\right\|^2_{L^\infty(M)}\,   \| f\|^2_{2,w} \nonumber \\
& \leq c_1 ( \| f\|_{2,w}^2+ \|\nabla f\|_{2,w}^2)\, . 
\end{align*}
for some $c_1$. It follows that $u\in H^1(M, dx)$ and 
\begin{equation} \label{h1-u}
\|u \|_{H^1(M,dx)}^2 \, \leq \, (c_1+c_0)\,  \|f\|^2_{H^1(M,w^2\, dx)}\, .
\end{equation}
On the other hand, the standard Sobolev imbedding theorem, see e.g.~\cite[Thm.~4.12]{ad}, says that there exists some $K_{q}$ such that 
$$
 \|u\|_{L^q(M,dx)} \, \leq \, K_q\, \|u \|_{H^1(M,dx)}\, .
$$
Since  $\|f\|_{q,w}  =  \|u\|_{L^q(M,dx)} $, the claim follows from \eqref{h1-u}.
\end{proof}

%%%%%%%%%%%%%%%%%%%%%%%%%%%
\section*{\bf Acknowledgements}
H.~K. was supported by the Gruppo Nazionale per Analisi Matematica, la Probabilit\`a e le loro Applicazioni (GNAMPA) of the Istituto Nazionale di Alta Matematica (INdAM). The support of MIUR-PRIN2010-11 grant for the project  ``Calcolo delle variazioni'' (H.~K.) is also gratefully acknowledged.

\end{document}